\documentclass[12pt]{amsart}
\usepackage{amssymb}
\usepackage{amsfonts}

\newtheorem{theorem}{Theorem}
\newtheorem{lemma}{Lemma}

\theoremstyle{definition}

\newtheorem{prop}{Proposition}
\newtheorem{cor}[theorem]{Corollary}

\theoremstyle{remark}

\begin{document}

\author{E. Liflyand}

\title {Integrability of the Fourier transform: functions of bounded variation}

\subjclass{Primary 42A38; Secondary 42A50}

\keywords{Fourier transform, integrability, Hilbert transform, Hardy
space}

\address{Department of Mathematics, Bar-Ilan University, 52900 Ramat-Gan, Israel}
\email{liflyand@math.biu.ac.il}

\begin{abstract}
Certain relations between the Fourier transform of a function of
bounded variation and the Hilbert transform of its derivative are
revealed. The widest subspaces of the space of functions of bounded
variation are indicated in which the cosine and sine Fourier
transforms are integrable.

\end{abstract}

\maketitle

\section{Introduction}

We are going to compare the Fourier transform of a function of
bounded variation and the Hilbert transform of a related function.
For this, let us start with some known results. The first one is
given in \cite[Thm.2]{L0} (see also \cite{Fr}). We define the
following $T$-transform of a function $g: \mathbb R=[0,\infty)\to
\mathbb C$:

\begin{eqnarray*} Tg(t)=\int_0^{t/2}\frac{g(t+s)-g(t-s)}{s}\,ds,\end{eqnarray*}
where the integral is understood in the improper (principal value)
sense, that is, as $\lim\limits_{\delta\to0+}\int_\delta.$

\begin{theorem}\label{ft} Let $f: \mathbb R_+\to\mathbb C$ be
locally absolutely continuous, of bounded variation and
$\lim\limits_{t\to\infty}f(t)=0.$ Let also $Tf'\in L^1(\mathbb
R_+).$ Then the cosine Fourier transform of $f$

\begin{eqnarray}\label{fc} \widehat{f_c}(x)=\int_0^\infty f(t)\cos
xt\,dt \end{eqnarray} is Lebesgue integrable on $\mathbb R_+$, with

\begin{eqnarray}\label{fce}\|\widehat{f_c}\|_{L^1(\mathbb R_+)}\lesssim
\|f'\|_{L^1(\mathbb R_+)}+\|Tf'\|_{L^1(\mathbb R_+)},\end{eqnarray}
and for the sine Fourier transform, we have, with $x>0,$

\begin{eqnarray}\label{fs} \widehat{f_s}(x)=\int_0^\infty f(t)\sin
xt\,dt=\frac{1}{x}f\left(\frac{\pi}{2x}\right)+F(x),
\end{eqnarray}
where

\begin{eqnarray}\label{fse}\|F\|_{L^1(\mathbb R_+)}\lesssim
\|f'\|_{L^1(\mathbb R_+)}+\|Tf'\|_{L^1(\mathbb R_+)}.\end{eqnarray}
\end{theorem}

Here and in what follows we use the notation ``$\, \lesssim \, $''
and ``$\, \gtrsim \, $'' as abbreviations for ``$\, \le C \, $'' and
``$\, \ge C \, $'', with $C$ being an absolute positive constant.

Let us now turn to the Hilbert transform of an integrable function
$g$

\begin{eqnarray}\label{dht}\mathcal{H}g(x)=\frac{1}{\pi}\int_\mathbb{R}\frac{g(t)}{t-x}\,dt,\end{eqnarray}
where the integral is also understood in the improper (principal
value) sense, now as $\lim\limits_{\delta\to0+}\int_{|t-x|>\delta}.$
It is not necessarily integrable, and when it is, we say that $g$ is
in the (real) Hardy space $H^1(\mathbb R).$ If $g\in H^1(\mathbb
R)$, then

\begin{eqnarray}\label{vm}\int_{\mathbb R}  g(t)\,dt=0.\end{eqnarray}
It was apparently first mentioned in \cite{kober}.

An odd function always satisfies (\ref{vm}). However, not every odd
integrable function belongs to $H^1(\mathbb R)$, for a
counterexample see, e.g., \cite{LiTi0}. When in the definition of
the Hilbert transform (\ref{dht}) the function $g$ is odd, we will
denote this transform by $\mathcal{H}_0,$ and it is equal to

\begin{eqnarray}\label{dht0}\mathcal{H}_0g(x)=\frac{2}{\pi}\int_0^\infty\frac{tg(t)}{t^2-x^2}\,dt.\end{eqnarray}
If it is integrable, we will denote the corresponding Hardy space by
$H_0^1(\mathbb R)$.

Since

\begin{eqnarray*} {\mathcal{H}_0}g(x)=Tg(x)+\Gamma(x),\end{eqnarray*}
where $\Gamma$ is such that

\begin{eqnarray*}\int_0^\infty |\Gamma(x)|\,dx\lesssim\int_0^\infty|g(t)|\,dt,\end{eqnarray*}
the right-hand sides of (\ref{fce}) and (\ref{fse}) can be treated
as $\|f'\|_{H_0^1(\mathbb R_+)}.$ This has been observed in
\cite{L0} and later on in \cite{Fr}.

The space of integrable functions $g$ with integrable $Tg$, or just
$H_0^1(\mathbb R_+),$ is one of the widest spaces the belonging of
the derivative $f'$ to which ensures the integrability of the cosine
Fourier transform of $f.$ However, the possibility of existence (or
non-existence) of a wider space of such type is of considerable
interest. Let us show that such a space does exist, moreover, it is
the widest possible, at least provides a necessary and sufficient
condition for the integrability of the cosine Fourier transform. In
fact, it has in essence been introduced (for different purposes) in
\cite{JW} as

\begin{eqnarray}\label{spQ}Q=\{g: g\in L^1(\mathbb R), \int_{\mathbb
R}\frac{|\widehat{g}(x)|}{|x|}\,dx<\infty\}.       \end{eqnarray}
With the obvious norm

\begin{eqnarray*} \|g\|_{L^1(\mathbb R)}+\int_{\mathbb
R}\frac{|\widehat{g}(x)|}{|x|}\,dx                 \end{eqnarray*}
it is a Banach space and ideal in $L^1(\mathbb R).$  What we will
actually use is the space $Q_0$ of the odd functions from $Q$

\begin{eqnarray}\label{spQ0}Q_0=\{g: g\in L^1(\mathbb R), g(-t)=-g(t),
\int_0^\infty\frac{|\widehat{g_s}(x)|}{x}\,dx<\infty\};\end{eqnarray}
such functions naturally satisfy (\ref{vm}).

\begin{theorem}\label{wider} Let $f: \mathbb R_+\to\mathbb C$ be
locally absolutely continuous, of bounded variation and
$\lim\limits_{t\to\infty}f(t)=0.$ Then the cosine Fourier transform
of $f$ given by (\ref{fc}) is Lebesgue integrable on $\mathbb R_+$
if and only if $f'\in Q_0$.
\end{theorem}

The situation is more delicate with the sine Fourier transform,
where a sort of asymptotic relation can be obtained. In what follows
we shall denote

\begin{eqnarray}\label{tq}\mathcal{T}_g(x)=\frac{\widehat{g_s}(x)}{x}.
\end{eqnarray}

\begin{theorem}\label{widers} Let $f: \mathbb R_+\to\mathbb C$ be
locally absolutely continuous, of bounded variation and
$\lim\limits_{t\to\infty}f(t)=0.$ Then for the sine Fourier
transform of $f$ given in (\ref{fs}) there holds for any $x>0$

\begin{eqnarray}\label{fsg} \widehat{f_s}(x)=\int_0^\infty f(t)\sin
xt\,dt=\frac{1}{x}f\left(\frac{\pi}{2x}\right)+ {\mathcal
H}_0{\mathcal T}_{f'}(x)+G(x),
\end{eqnarray}
where

\begin{eqnarray}\label{fseg}\|G\|_{L^1(\mathbb R_+)}\lesssim
\|f'\|_{L^1(\mathbb R_+)}.\end{eqnarray}
\end{theorem}

This theorem makes it natural to consider a Hardy type space
$H^1_Q(\mathbb R_+)$ which consists of $Q_0$ functions $g$ with
integrable ${\mathcal H}_0{\mathcal T}_g.$

\begin{cor} Let a function $f$ satisfy the assumptions of Theorem
\ref{widers} and such that $f'\in H^1_Q(\mathbb R_+).$ Then

\begin{eqnarray}\label{fsgc} \widehat{f_s}(x)=\frac{1}{x}f\left(\frac{\pi}{2x}\right)+G(x),
\end{eqnarray}
where

\begin{eqnarray}\label{fseg}\|G\|_{L^1(\mathbb R_+)}\lesssim
\|f'\|_{H_Q^1(\mathbb R_+)}.\end{eqnarray}
\end{cor}
Technically, this is an obvious corollary of Theorem \ref{widers}.
We shall discuss it in Section 3.

\bigskip

\section{Proofs}

\begin{proof}[Proof of Theorem \ref{wider}] The assumptions of the theorem give a possibility to
integrate by parts. This yields

\begin{eqnarray*} \widehat{f_c}(x)=-\frac{1}{x}\int_0^\infty
f'(t)\sin xt\,dt=-\frac{1}{x}\widehat{f'_s}(x).\end{eqnarray*}
Integrating both sides over $\mathbb R_+$ completes the proof.
 \hfill\end{proof}

\begin{proof}[Proof of Theorem \ref{widers}] Let us start with integration by parts in

\begin{eqnarray*}\int_0^{\frac{\pi}{2x}} f(t)\sin xt\,dt&=&
\frac{1-\cos xt}{x}f(t)\mid_0^{\frac{\pi}{2x}}+
\frac{1}{x}\int_0^{\frac{\pi}{2x}} f'(t)[\cos xt-1]\,dt\\
&=&\frac{1}{x}f\left(\frac{\pi}{2x}\right)+
\frac{1}{x}\int_0^{\frac{\pi}{2x}} f'(t)[\cos xt-1]\,dt.
\end{eqnarray*}
The last value is bounded by $\int_0^{\frac{\pi}{2x}} t|f'(t)|\,dt$,
and

\begin{eqnarray}\label{sm}\int_0^\infty\,\int_0^{\frac{\pi}{2x}}
t|f'(t)|\,dt\,dx=\frac{\pi}{2}\int_0^\infty|f'(t)|\,dt.\end{eqnarray}

Let us now consider

\begin{eqnarray}\label{mainint}I=I(x)=\int_{\frac{\pi}{2x}}^\infty
f(t)\sin xt\,dt.                               \end{eqnarray}

We will start with the following statement.

\begin{lemma}\label{iht} There holds

\begin{eqnarray*}{\mathcal H}_0{\mathcal T}_{f'}(x)=\frac{2}{x\pi}
\int_0^\infty f'(t)\sin xt\,\int_{xt}^\infty\frac{\cos
v}{v}\,dv\,dt+\frac{2}{x\pi} \int_0^\infty f'(t)\cos
xt\,\int_0^{xt}\frac{\sin v}{v}\,dv\,dt.\end{eqnarray*}
\end{lemma}

{\bf Proof of Lemma \ref{iht}.} We have

\begin{eqnarray}\label{sfs}{\mathcal H}_0{\mathcal T}_{f'}(x)&=&\frac{2}{\pi}
\int_0^\infty \widehat{f'_s}(u)\,\frac{1}{u^2-x^2}\,du\\
&=&\frac{2}{\pi}\int_0^\infty f'(t)\,\int_0^\infty\frac{1}{u^2-x^2}
\sin ut\,du\,dt.                                   \end{eqnarray}
Denoting, as usual,

\begin{eqnarray*} {\rm Ci}(u)=-\int_u^\infty\frac{\cos
t}{t}\,dt\end{eqnarray*}                 and

\begin{eqnarray*} {\rm Si}(u)=\int_0^u\frac{\sin t}{t}\,dt =
\frac{\pi}{2}-\int_u^\infty\frac{\sin t}{t}\,dt,   \end{eqnarray*}
we will make use of the formula (see \cite[Ch.II, (18)]{BE})

\begin{eqnarray*}\int_0^\infty \frac{1}{a^2-x^2}\,\sin yx\,dx=
\frac1a[\sin ay\,{\rm Ci}(ay)-\cos ay\,{\rm Si(ay)}],\end{eqnarray*}
where the integrals is understood in the principal value sense and
$a,y>0.$ We apply this formula to the inner integral on the
right-hand side of (\ref{sfs}), with $a=x$ and $y=t.$ Using the
expressions for ${\rm Ci}$ and ${\rm Si}$, we complete the proof of
the lemma. \hfill$\Box$

With this in hand, we are going to prove that

\begin{eqnarray}\label{ihtf}I(x)={\mathcal H}_0{\mathcal T}_{f'}(x)+G(x),\end{eqnarray}
where

\begin{eqnarray}\label{fseg1}\|G\|_{L^1(\mathbb R_+)}\lesssim
\|f'\|_{L^1(\mathbb R_+)}.                          \end{eqnarray}
We denote the two summands in the expression obtained in the lemma
by $I_1$ and $I_2.$ For both, we make use of the fact that

\begin{eqnarray*}\int_{xt}^\infty\frac{\cos v}{v}\,dv=O(\frac{1}{xt}).\end{eqnarray*}
The same true when $\cos v$ is replaced by $\sin v.$ When
$t\ge\frac{1}{x}$ we have

\begin{eqnarray}\label{big}\int_0^\infty\frac{1}{x}\int_{\frac{1}{x}}^\infty|f'(t)|\frac{1}{xt}\,dt\,dx =
\int_0^\infty|f'(t)|t\,\int_{\frac{1}{t}}^\infty\frac{1}{x^2}\,dx=
\int_0^\infty |f'(t)|\,dt.\end{eqnarray}

When $t\le\frac{1}{x}$ we split the inner integral into two. First,

\begin{eqnarray*}\int_1^\infty\frac{\cos v}{v}\,dv=O(1),\end{eqnarray*}
and using $\biggl|\frac{\sin xt}{x}\biggr|\le t,$ we arrive at the
estimate similar to (\ref{sm}). Further, we have

\begin{eqnarray*}\int_{xt}^1\biggl|\frac{\cos v}{v}\biggr|\,dv=O(\ln\frac{1}{xt}).\end{eqnarray*}
By this, integrating in $x$ over $(0,\infty),$ we have to estimate

\begin{eqnarray*}\int_0^\infty|f'(t)|t\,\int_0^{1/t}\ln\frac{1}{xt}\,dx\,dt
=\int_0^\infty|f'(t)|\,dt.\end{eqnarray*} Here we use that

\begin{eqnarray*}\int_0^{1/t}\ln\frac{1}{xt}\,dx=\frac{1}{t}.\end{eqnarray*}
In conclusion, $I_1$ can be treated as $G$.

Let us proceed to $I_2.$ Using that

\begin{eqnarray*}\frac{1}{x}\biggl|\int_0^{xt}\frac{\sin v}{v}\,dv\biggr|=O(t),\end{eqnarray*}
we arrive for $t\le\frac{\pi}{2x}$ at (\ref{sm}). Let now
$t\le\frac{\pi}{2x}$. We have

\begin{eqnarray*}\int_0^{xt}\frac{\sin v}{v}\,dv=\frac{\pi}{2}-
\int_{xt}^\infty\frac{\sin v}{v}\,dv.\end{eqnarray*}  Now

\begin{eqnarray*}\frac{2}{x\pi} \int_0^\infty f'(t)\cos
xt\,dt\frac{\pi}{2}=I.\end{eqnarray*} For the integral
$\int_{xt}^\infty\frac{\sin v}{v}\,dv$, the estimates are exactly
like those in (\ref{big}).

Combining (\ref{ihtf}) and estimates before Lemma \ref{iht}, we
complete the proof of the theorem.  \hfill\end{proof}

\bigskip

\section{Discussion}

Discussion and comments are in order. At first sight, Theorem
\ref{wider} does not seem to be a result at all, at most a technical
reformulation of (\ref{fc}). This could be so but not after the
appearance of the analysis of $Q$ in \cite{JW}. Indeed, the
well-known extension of Hardy's inequality (see, e.g.,
\cite[(7.24)]{GR})

\begin{eqnarray}\label{Fein}\int_{\mathbb
R}\frac{|\widehat{g}(x)|}{|x|}\,dx\lesssim\|g\|_{H^1(\mathbb R)}
\end{eqnarray}
implies

\begin{eqnarray}\label{em1}H^1(\mathbb R)\subseteq Q\subseteq L^1_0(\mathbb R),\end{eqnarray}
where the latter is the subspace of $g$ in $L^1(\mathbb R)$ which
satisfy the cancelation property (\ref{vm}).

It is worth noting that (\ref{em1}) immediately proves (\ref{fce})
from Theorem \ref{ft}. The initial proof in \cite{L0} is essentially
more complicated.

It is doubtful that $Q$ (or $Q_0$) may be defined in terms of $f$
itself rather than its Fourier transform, therefore it is of
interest to find certain proper subspaces of $Q_0$ wider than $H^1$
belonging to which is easily verifiable. We mention the paper
\cite{sweezy} in which a family of subspaces between $H^1$ and $L^1$
is introduced and duality properties of that family are studied.
However, it is not clear how to compare that family with $Q_0.$

Back to Theorem \ref{widers}, let us analyze (\ref{ihtf}). On the
one hand, we have

\begin{eqnarray*} \int_0^\infty |I(x)|\,dx=\int_0^\infty |{\mathcal
H}_0{\mathcal T}_{f'}(x)|\,dx+O(\|f'\|_{H^1_0(\mathbb R_+)}).
\end{eqnarray*}
On the other hand, it is proved in \cite{L0} that

\begin{eqnarray*} \int_0^\infty |I(x)|\,dx=O(\|f'\|_{H^1_0(\mathbb R_+)}).
\end{eqnarray*}
This leads to

\begin{prop}\label{comph} If $g$ is an integrable odd function, then

\begin{eqnarray*} \|{\mathcal H}_0{\mathcal T}_g\|_{L^1(\mathbb R_+)}
\lesssim \|g\|_{H_0^1(\mathbb R_+)}.                 \end{eqnarray*}
\end{prop}

The above proof of Proposition \ref{comph} looks "artificial". A
direct proof, preferable simple enough will be very desirable. In
any case, this implies an updated chain of embeddings

\begin{eqnarray}\label{em2}H_0^1(\mathbb R_+)\subseteq H^1_Q(\mathbb R_+)\subseteq
Q_0\subseteq L^1_0(\mathbb R_+).                     \end{eqnarray}
It is very interesting to figure out which of these embeddings are
proper. Correspondingly, intermediate spaces are of interest, both
theoretical and practical.

\end{document}